\documentclass[a4paper,12pt]{amsart}
\usepackage{amssymb}
\usepackage[all,cmtip]{xy}
\usepackage{amsmath}
\usepackage{chemarrow}
\usepackage[colorlinks, 
pdfborder=001, 
linkcolor=green,
anchorcolor=green,
citecolor=green
]{hyperref}
\usepackage{mathrsfs}
\usepackage{wasysym}
\usepackage{titletoc}
\usepackage{bm}
\usepackage{geometry}
\usepackage{colordvi}
\usepackage{color}

\allowdisplaybreaks

\DeclareMathOperator{\Aut}{Aut}

\DeclareMathOperator{\GKdim}{GKdim}

\DeclareMathOperator{\kk}{\Bbbk}

\DeclareMathOperator{\N}{\mathbb{N}}

\DeclareMathOperator{\Z}{\mathbb{Z}}

\numberwithin{equation}{section}

\theoremstyle{definition}

\newtheorem{thm}{Theorem}[section]
\newtheorem{prop}[thm]{Proposition}
\newtheorem{lem}[thm]{Lemma}

\newtheorem{conj}[thm]{Conjecture}
\newtheorem{defn}[thm]{Definition}

\begin{document}

\title{Some invariant subalgebras 
are graded isolated singularities}


\author{Ruipeng Zhu}
\address{Department of Mathematics, Southern University of Science and Technology, Shenzhen, Guangdong 518055, China}
\email{zhurp@sustech.edu.cn}

\begin{abstract}
	In this note, we prove that the invariant subalgebra of the skew polynomial algebra $\kk \langle x_0, x_1, \cdots, x_{n-1} \rangle / (\{x_ix_j+x_jx_i \mid i \neq j\})$ under the action $x_i \mapsto x_{i+1}\,(i \in \Z/{n\Z})$ is a graded isolated singularity, and thus a conjecture of Chan-Young-Zhang is true.
\end{abstract}


\thanks{}

\maketitle



\section{Introduction}
Noncommutative graded isolated singularities are defined by Ueyama \cite[Definition 2.2]{Uey2013}.
A noetherian connected graded algebra $B$ is called a {\it graded isolated singularity} if the associated noncommutative projective scheme $\mathrm{Proj} (B)$ (in the sense of \cite{AZ1994}) has finite global dimension. See \cite{CKWZ2018, BHZ2018, GKMW2019, CKZ2020} for some examples of graded isolated singularities.

Let $A$ be a noetherian Artin-Schelter regular algebra and $G$ be a finite subgroup of $\Aut_{\mathrm{gr}}(A)$. To prove  a version of the noncommutative Auslander theorem, an invariant called the {\it pertinency} of the $G$-action on $A$  is introduced in \cite{BHZ2018} and \cite{BHZ2019}. We recall it here.

The {\it pertinency} of the $G$-action on $A$ \cite[Definition 0.1]{BHZ2019} is defined to be
$$\mathbf{p}(A, G):= \GKdim(A) - \GKdim(A\#G/(e_0)),$$
where $(e_0)$ is the ideal of the skew group algebra $A\#G$ generated by $e_0 := 1 \# \frac{1}{|G|} \sum_{g \in G} g$.

Then, by \cite[Theorem 3.10]{MU2016}, $A^G$ is a graded isolated singularity if and only if $\mathbf{p}(A, G) = \GKdim(A)$. Unlike in the commutative cases, it is difficult to determine when the invariant subalgebra is a graded isolated singularity. 

Let $\kk$ be an  algebraically closed field of characteristic zero. 
Let $A = \kk_{-1}[x_0, \dots, x_{n-1}] \,(n \geqslant 2)$  be the ($-1$)-skew polynomial algebra, which is generated by $\{x_0, \dots, x_{n-1}\}$ and subject to the relations
$$x_ix_j = (-1) x_jx_i\,(\forall i \neq j).$$
Let $G:=C_n$ be the cyclic group of order $n$ acting on $A$ by permuting the generators of the algebra cyclically; namely, $C_n$ is generated by $\sigma = (0 \, 1 \, 2 \, \cdots \, n-1)$ of order $n$ that acts on the generators by
$$\sigma x_i = x_{i+1}, \; \forall \, i \in \Z_n:= \Z/n\Z.$$

In \cite[Theorem 0.4]{CYZ2020}, Chan, Young and Zhang prove the following result on graded isolated singularities.

\begin{thm}\label{CYZ-thm}
	If either $3$ or $5$ divides $n$, then $\mathbf{p}(A, G) < \GKdim A = n$. Consequently, the invariant subalgebra $A^G$ is not a graded isolated singularity.
\end{thm}

Based on this theorem and \cite[Theorem 0.2]{CYZ2020}, Chan, Young and Zhang give the following conjecture \cite[Conjecture 0.5]{CYZ2020}.

\begin{conj}\label{conj}
	The invariant subalgebra $A^G$ is a graded isolated singularity if and only if $n$ is not divisible by $3$ or $5$.
\end{conj}

To prove Conjecture \ref{conj} is true, it suffices to prove the following theorem, which is the main result in this note.

\begin{thm}\label{main-thm}
	If $n$ is not divisible by $3$ or $5$, then $\mathbf{p}(A, G) = \GKdim A = n$. As a consequence, $A^G$ is a graded isolated singularity.
\end{thm}


\section{Preliminaries}

Before giving a proof of Theorem \ref{main-thm}, let us recall some notations and results in \cite{CYZ2020}. 

Let $\omega$ be a primitive $n$th root of unity. For any $\gamma = 0, 1, \dots, n-1 \in \Z_n$, let
$$b_{\gamma} := \frac{1}{n} \sum_{i=0}^{n-1} \omega^{i\gamma} x_i \in A \subseteq A \# C_n.$$
Then $b_{\gamma}$ is an $\omega^{- \gamma}$-eigenvector of $\sigma$. 
Let
$$e_{\gamma} := \frac{1}{n} \sum_{i=0}^{n-1} (\omega^{\gamma} \sigma)^{i} \in \kk C_n \subseteq A\#C_n,$$
which are idempotent elements.

Suppose $\deg(x_i) = 1$ and $\deg(e_i) = 0$ for all $i \in \Z_n$. 
As usual, $[-,-]$ denotes the graded commutator of the graded ring $A\#C_n$, that is, $[u,v] = uv - (-1)^{\deg(u)\deg(v)} vu$ for any homogeneous elements $u,v \in A\#C_n$. 

\begin{lem}\cite[Lemma 1.1]{CYZ2020}
	The graded algebras $A$ and $A\#C_n$ can be presented as
	$$A \cong \frac{\kk \langle b_0, \dots, b_{n-1} \rangle}{([b_0, b_k] - [b_l, b_{k-l}])}\, \text{ and }\, A \# C_n \cong \frac{\kk \langle b_0, \dots, b_{n-1}, e_0, \dots, e_{n-1} \rangle}{(e_{\alpha}b_{\gamma} - b_{\gamma}e_{\alpha-\gamma}, e_ie_j-\delta_{ij}e_i, [b_0, b_k] - [b_l, b_{k-l}])}$$
	respectively, where $\delta_{ij}$ is the Kronecker delta and indices are taken modulo $n$.
\end{lem}

For each $j \in \Z_n$, let
$$c_j := [b_k, b_{j-k}] = b_kb_{j-k} + b_{j-k}b_k = \frac{2}{n^2}\sum_{i=0}^{n-1}\omega^{ij}x_i^2.$$
Then $c_j$ is an $\omega^{-j}$-eigenvector of $\sigma$.

For any vector $\mathbf{i} = (i_0, \dots, i_{n-1}) \in \N^n$, 
we use the following notations:
$$\mathbf{b^i} = b_0^{i_0} \cdots b_{n-1}^{i_{n-1}} \,\textrm{ and }\, \mathbf{c^i} = c_0^{i_0} \cdots c_{n-1}^{i_{n-1}}.$$


Let $R_{\gamma}$ be the subspace of $A$ spanned by 
the elements $\mathbf{b^ic^j}$ such that $\sum_{s=0}^{n-1} (i_s+j_s)s = \gamma \mod n$; that is, $R_{\gamma}$ consists of $\omega^{-\gamma}$-eigenvectors of $\sigma$. This gives an $R_0$-module decomposition
$$A = R_0 \oplus R_1 \oplus \cdots \oplus R_{n-1}.$$

\begin{defn}
	\begin{enumerate}
		\item Let $\Phi_n:= \{ k \mid c_k^{N_k} \in (e_0) \text{ for some } N_k \geq 0 \}$, where $(e_0)$ is the two-sided ideal of $A\#C_n$ containing $e_0$.
		\item Let $\phi_2(n) := \{ k \mid 0 \leq k \leq n-1, \gcd(k, n) = 2^{w} \text{ for some } w \geq 0 \}$.
		\item Let $\Psi_{j}^{[n]}:= \{ i \mid c_i^N \in R_jA \text{ for some } N \geq 0 \}$.
		\item \cite[Definition 5.2 and Lemma 5.3(1)]{CYZ2020} We say $n$ is {\it admissible} if, for any $i$ and $j$, $i \in \Psi_j^{[n]}$, or equivalently, $\GKdim (A\#C_n/(e_0)) = 0$.
		\item Let $\overline{A} := A/\langle c_k \mid k \in \Phi_n \rangle$, and $\overline{\Psi}_{j}^{[n]}:= \{ i \mid c_i^N \in \overline{R}_j\overline{A} \text{ for some } N \geq 0 \}$ where $\overline{R}_j = \frac{R_j + \langle c_k \mid k \in \Phi_n \rangle}{\langle c_k \mid k \in \Phi_n \rangle} \subseteq \overline{A}$.
	\end{enumerate}
\end{defn}


Let $\Z_n^{\times}$ be the set of invertible elements in $\Z_n$.

\begin{lem}\label{lem1}
	\begin{enumerate}
		\item \cite[Definition 6.1]{CYZ2020} $\Phi_n$ is a special subset of $\Z_n$, that is, $k \in \Phi_n$ if and only if $\lambda k \in \Phi_n$ for all $\lambda \in \Z_n^{\times}$.
		\item \cite[Proposition 2.3]{CYZ2020} $\phi_2(n) \subseteq \Phi_n$.
	\end{enumerate}
\end{lem}

The following proposition follows from the proof of \cite[Proposition 6.6]{CYZ2020}.

\begin{prop}\label{prop1}
	Let $n \geq 2$ such that $3, 5 \nmid n$. If $1, \dots, n-1 \in \overline{\Psi}_1^{[n]}$, then $0 \in \overline{\Psi}_1^{[n]}$.
\end{prop}

\begin{prop}\cite[Proposition 6.8]{CYZ2020}\label{prop2}
	Let $n \geq 2$. Suppose that
	\begin{enumerate}
		\item every proper factor of $n$ is admissible, and
		\item for each $0 \leq i \leq n-1$, $i \in \overline{\Psi}_1^{[n]}$.
	\end{enumerate}
    Then $n$ is admissible.
\end{prop}

\section{Proof of the Theorem \ref{main-thm}}\label{sec2}

\begin{proof}[Proof of Theorem \ref{main-thm}]
	We prove it by induction on $n$. Assume that every proper factor of $n$ is admissible.
	By Proposition \ref{prop2}, it suffices to prove that
	$$\text{ for each } 0 \leq i \leq n-1, \; i \in \overline{\Psi}_1^{[n]}.$$
	If this is not true, that is, there is $ 0 \leq m \leq n-1$ such that $m \notin \overline{\Psi}_1^{[n]}$. Then we may assume that
	\begin{enumerate}
		\item $m \neq 0$, by Proposition \ref{prop1};
		\item $m \mid n$, by Lemma \ref{lem1} (2) as $\Z_n^{\times} \subseteq \phi_2(n) \subseteq \Phi_n$;
		\item $m > 5$, by Lemma \ref{lem1} (2) and assumption $3,5 \nmid n$.
	\end{enumerate}
    Write $n = mq$ with $q>1$.
	
	Since $c_m$ is an eigenvector of $\sigma$, then $C_n$ acts on the localization $A[c_m^{-1}]$, and $A[c_m^{-1}] \# C_n / (e_0) \cong (A \# C_n / (e_0))[c_m^{-1}]$.
    Let
    $$\widetilde{A} = \frac{\kk \langle b_0, \dots, b_{m-1} \rangle}{([b_0, b_k] - [b_l, b_{k-l}] \mid l, k \in \Z_m)}$$
    be a subalgebra of $A$,
    and $\widetilde{R}_{\gamma}$ be the subspace of $\widetilde{A}$ spanned by the elements $\mathbf{b^ic^j}$  such that $\sum_{s=0}^{m-1} (i_s+j_s)s = \gamma \mod m$.
    For any $\mathbf{b^ic^j} = b_0^{i_0} \cdots b_{m-1}^{i_{m-1}}c_0^{j_0} \cdots c_{m-1}^{j_{m-1}} \in \widetilde{R}_1$ with
    $$\sum_{s=0}^{m-1} (i_s + j_s)s = mk + 1 \text{ for some } k \geq 0,$$
    then there exists $l > 0$ such that $(l-1)q \leq k < lq$.
    Hence $\mathbf{b^ic^j}c_{m}^{lq-k} \in R_1A$, and $\mathbf{b^ic^j} \in R_1A[c_m^{-1}]$. It follows that
    $$\widetilde{R}_1\widetilde{A} \subseteq R_1A[c_m^{-1}].$$
    
    Write $\widetilde{\omega} = \omega^q$. Note that $\widetilde{A} \cong \kk_{-1}[\widetilde{x}_0, \dots, \widetilde{x}_{m-1}]$ via $b_{\gamma} \mapsto \frac{1}{m} \sum_{i=0}^{m-1}\widetilde{\omega}^{i \gamma} \widetilde{x_i}$. Then the cyclic group $C_m$ of order $m$ acts on $\widetilde{A}$ by permuting the generators of the algebra cyclically; namely, $C_m$ is generated by $\widetilde{\sigma} = (012 \cdots m-1)$ of order $m$ that acts on the generators by
    $$\widetilde{\sigma} \widetilde{x_i} = \widetilde{x_{i+1}}, \; \forall \, i \in \Z_m.$$
    Then $\widetilde{R}_{\gamma}$ consists of $\widetilde{\omega}^{-\gamma}$-eigenvectors of $\widetilde{\sigma}$.
    By assumption, $m$ is admissible, so for any $0 \leq i \leq m-1$, there exists $N_i$ such that
    $$c_i^{N_i} \in \widetilde{R}_1\widetilde{A} \subseteq R_1A[c_m^{-1}].$$
    
    Let $\Gamma$ be the right ideal $R_1A[c_m^{-1}] + \sum\limits_{\exists \, N_k, \, {c_k^{N_k}} \in R_1A[c_m^{-1}]} c_kA[c_m^{-1}]$ of $A[c_m^{-1}]$. Next we prove that $\Gamma = A[c_m^{-1}]$.
    The following proof is quite similar to the proof of \cite[Porposition 6.6]{CYZ2020}.
    
    ~\\
    
    Claim 1. Let $0 \leq j < \frac{m-1}{2}$. If $c_m^sb_j \in \Gamma$ for some $s>0$, then $c_m^{s+1}b_{j+1} \in \Gamma$.
    \begin{proof}[Proof of Claim 1]
    	First of all, $b_{j+1}b_{m-j} \in \widetilde{R}_1\widetilde{A} \subseteq R_1A[c_m^{-1}]$ since $(j+1) + (m-j) = 1 \mod m$.
    	Due to $c_m^sb_j \in \Gamma$, then
    	\begin{align*}
    		 \Gamma \ni & [b_{j+1}b_{m-j}, c_m^sb_j] \\
    		& = c_m^sb_{j+1}b_{m-j}b_j - c_m^sb_jb_{j+1}b_{m-j} \\
    		& = c_m^{s}b_{j+1}b_{m-j}b_j + c_m^sb_{j+1}b_jb_{m-j} - c_m^sc_{2j+1}b_{m-j} \\
    		& = c_m^{s+1}b_{j+1} - c_m^sc_{2j+1}b_{m-j}.
    	\end{align*}
    Since there exists $N_{2j+1} > 0$ such that $c_{2j+1}^{N_{2j+1}} \in \widetilde{R}_1\widetilde{A}$ for $2j+1 < m$, then $c_m^{s+1}b_{j+1} \in \Gamma$.
    \end{proof}
    ~\\
    
    Claim 2. Suppose that $m = 2k +1$. If $c_m^sb_k \in \Gamma$, then $c_m^{s+2}b_{k+2} \in \Gamma$.
    \begin{proof}[Proof of Claim 2]
    	Note that $b_{k+1}b_{k+2}b_{m-1} \in \widetilde{R}_1\widetilde{A} \subseteq \Gamma$ as $(k+1) + (k+2) + (m-1) = 1 \mod m$.
    	\begin{align*}
    		\Gamma \ni & [c_m^sb_k, b_{k+1}b_{k+2}b_{m-1}] \\
    		& = c_m^sb_kb_{k+1}b_{k+2}b_{m-1} + c_m^sb_{k+1}b_{k+2}b_{m-1}b_k \\
    		& = c_m^sb_kb_{k+1}b_{k+2}b_{m-1} + c_m^sc_{3k} b_{k+1}b_{k+2} - c_m^sb_{k+1}b_{k+2}b_kb_{m-1} \\
    		& = c_m^sb_kb_{k+1}b_{k+2}b_{m-1} + c_m^sc_{3k} b_{k+1}b_{k+2} - c_m^sc_{m+1}b_{k+1}b_{m-1} + c_m^sb_{k+1}b_kb_{k+2}b_{m-1} \\
    		& = c_m^{s+1} b_{k+2}b_{m-1} + c_m^sc_{3k} b_{k+1}b_{k+2} - c_m^sc_{m+1}b_{k+1}b_{m-1}.
    	\end{align*}
    	Since $c_{m+1}c_m^{q-1} \in R_1 A$, then $c_{m+1} \in R_1A[c_m^{-1}]$. Hence $c_m^{s+1} b_{k+2}b_{m-1} + c_m^sc_{3k} b_{k+1}b_{k+2} \in \Gamma$.
    	\begin{align*}
    		\Gamma \ni & [c_m^{s+1} b_{k+2}b_{m-1} + c_m^sc_{3k} b_{k+1}b_{k+2}, b_1] \\
    		& = c_m^{s+1} b_{k+2}b_{m-1}b_1 - c_m^{s+1} b_1b_{k+2}b_{m-1} + c_m^sc_{3k} b_{k+1}b_{k+2}b_1 - c_m^sc_{3k} b_1b_{k+1}b_{k+2} \\
    		& = c_m^{s+2} b_{k+2} - c_m^{s+1} b_{k+2}b_1b_{m-1} - c_m^{s+1} b_1b_{k+2}b_{m-1} \\
    		& \;\; + c_m^sc_{3k} b_{k+1}b_{k+2}b_1 + c_m^sc_{3k} b_{k+1}b_1b_{k+2} - c_m^sc_{3k} c_{k+2}b_{k+2} \\
    		& = c_m^{s+2} b_{k+2} - c_m^{s+1} c_{k+3}b_{m-1} + c_m^sc_{3k} c_{k+3} b_{k+1} - c_m^sc_{3k} c_{k+2}b_{k+2}.
    	\end{align*}
    	By assumption $m > 5$, so $k > 2$. Since $k+2 < k+3 < 2k+1 = m$, $c_{k+2}, c_{k+3} \in \Gamma$ by assumption. It follows that $c_m^{s+2} b_{k+2} \in \Gamma$.
    \end{proof}
    ~\\
    
    Claim 3. $c_m^{m-1} \in \Gamma$.
    \begin{proof}[Proof of Claim 3]
    	 Assume that $m$ is even. Starting with $b_1$, and applying Claim 1 ($\frac{m}{2}-1$) times, we get $c_m^{\frac{m}{2}-1} b_{\frac{m}{2}} \in \Gamma$. Hence $c_m^{m-1} = [c_m^{\frac{m}{2}-1} b_{\frac{m}{2}}, c_m^{\frac{m}{2}-1} b_{\frac{m}{2}}] \in \Gamma$.
    	
    	If $m = 2k+1$ is odd, then by applying Claim 1 ($k-2$) and ($k-1$) times we get $c_m^{k-2} b_{k-1}$ and $c_m^{k-1} b_{k}\in \Gamma$ respectively. By applying Claim 2 we get $c_m^{k+1} b_{k+2} \in \Gamma$. 
    	
    	Therefore,
    	$c_m^{2k} = [c_m^{k-2} b_{k-1}, c_m^{k+1} b_{k+2}] \in \Gamma.$
    \end{proof}
    
    By Claim 3, $\Gamma = A[c_m^{-1}]$. Recall that $\Gamma = R_1A[c_m^{-1}] + \sum\limits_{\exists \, N_k, \, {c_k^{N_k}} \in R_1A[c_m^{-1}]} c_kA[c_m^{-1}]$. It is not difficult to see that $A[c_m^{-1}] = R_1A[c_m^{-1}]$.
    So there exists $N \geq 0$ such that $c_m^N \in R_1A$,
    which is a contradiction (as $m \notin \overline{\Psi}_1^{[n]}$). This implies $\overline{\Psi}_1^{[n]} = \{0, 1, \dotsm, n-1 \}$, that is, $n$ is admissible.
    Hence $\GKdim(A\#C_n / (e_0)) = 0$, and $\mathbf{p}(A, G) = n$.
\end{proof}

\section*{Acknowledgments} The author is very grateful to Professor Quanshui Wu and James Zhang who read the paper and made numerous helpful suggestions.

\bibliographystyle{siam}
\bibliography{References}

\end{document}